\documentclass[oneside,english]{amsart}
\usepackage[T1]{fontenc}
\usepackage[latin9]{inputenc}
\usepackage{amsthm}
\usepackage{amsbsy}
\usepackage{amstext}
\usepackage{amssymb}
\usepackage{graphicx}
\usepackage{xargs}[2008/03/08]

\makeatletter
\numberwithin{equation}{section}
\numberwithin{figure}{section}
  \theoremstyle{plain}
  \newtheorem*{thm*}{\protect\theoremname}
\theoremstyle{plain}
\newtheorem{thm}{\protect\theoremname}[section]
  \theoremstyle{plain}
  \newtheorem{prop}[thm]{\protect\propositionname}
  \theoremstyle{definition}
  \newtheorem{defn}[thm]{\protect\definitionname}
  \theoremstyle{remark}
  \newtheorem{rem}[thm]{\protect\remarkname}
  \theoremstyle{plain}
  \newtheorem{lem}[thm]{\protect\lemmaname}
  \theoremstyle{plain}
  \newtheorem{cor}[thm]{\protect\corollaryname}
  \theoremstyle{definition}
  \newtheorem{example}[thm]{\protect\examplename}

\makeatother

\usepackage{babel}
  \providecommand{\corollaryname}{Corollary}
  \providecommand{\definitionname}{Definition}
  \providecommand{\examplename}{Example}
  \providecommand{\lemmaname}{Lemma}
  \providecommand{\propositionname}{Proposition}
  \providecommand{\remarkname}{Remark}
  \providecommand{\theoremname}{Theorem}
\providecommand{\theoremname}{Theorem}

\begin{document}

\title{On the constancy regions for mixed test ideals}

\author{Felipe Pérez }
\begin{abstract}
In this note we study the partition of $\mathbb{R}_{\geq0}^{n}$ given
by the regions where the mixed test ideals $\tau(\mathfrak{a}_{1}^{t_{1}}...\mathfrak{a}_{n}^{t_{n}})$
are constant. We show that each region can be described as the preimage
of a natural number under a $p$-fractal function $\varphi:\mathbb{R}_{\geq0}^{n}\rightarrow\mathbb{N}$.
In addition, we give some examples illustrating that these regions
do not need to be composed of finitely many rational polytopes. 
\end{abstract}
\maketitle

\section{Introduction}

\newcommandx\mf[1][usedefault, addprefix=\global, 1=]{\mathfrak{#1}}

\newcommandx\bs[1][usedefault, addprefix=\global, 1=]{\boldsymbol{#1}}

In this note, we study the dependence of mixed test ideals on parameters,
and show that the emerging picture is quite different from that in
the case of mixed multiplier ideals in characteristic zero. 

Multiplier ideals have been intensively studied over the last two
decades, as they play an important role in birational geometry, see
for example \cite{Laz}. Given a smooth complex variety $X$ and a
nonzero ideal sheaf $\mathfrak{a}$, one can define for any parameter
$c>0$ an ideal $\mathcal{J}(\mathfrak{a}^{c})$, called multiplier
ideal. This ideal is described via a log resolution $\pi:X'\rightarrow X$
of the pair $(X,\mathfrak{a})$, i.e. a proper birational map, with
$X'$ smooth, and such that $\mathfrak{a}\mathcal{O}_{X'}=\mathcal{O}_{X'}(-E)$,
where $E$ is a simple normal crossing divisor. Then, 
\begin{equation}
\mathcal{J}(\mathfrak{a}^{c}):=\pi_{*}\mathcal{O}(K_{X'/X}-\lfloor cE\rfloor),
\end{equation}
where $K_{X'/X}$ is the relative canonical divisor. 

$\textit{Mixed multiplier ideals}$ extend the previous definition
to the case of several ideals: for nonzero ideals $\mathfrak{a}_{1},\ldots,\mathfrak{a_{n}}$
and positive numbers $c_{1},\ldots,c_{n}$ we take a log resolution
for the pair $(X,\mathfrak{a}_{1}\cdots\mathfrak{a}_{n})$ and set
the mixed multiplier ideal to be

\[
\mathcal{J}(\mathfrak{a}_{1}^{c_{1}}\cdots\mathfrak{a}_{n}^{c_{n}}):=\pi_{*}\mathcal{O}(K_{X'/X}-\lfloor c_{1}E_{1}+\ldots+c_{n}E_{n}\rfloor),
\]
where $\mathcal{O}_{X'}(-E_{i})=\mathfrak{a}_{i}\mathcal{O}_{X'}$. 

Test ideals were introduced by Hara and Yoshida in \cite{HY} as an
analogue of multiplier ideals in positive characteristic. One question
that was studied since \cite{HY} is which properties of multiplier
ideals have analogues for test ideals. For example, for multiplier
ideals the $\textit{jumping}$ $\textit{numbers}$ of $\mf[a]$ are
defined as the positive real numbers $c$ such that $\mathcal{J}(\mf[a]^{c})\not=\mathcal{J}(\mf[a]^{c-\epsilon})$
for every $\epsilon>0$ (cf. \cite{ELSV}). It is easy to see from
the definition (1.1) that for each $\mf[a]$ these numbers are discrete
and rational. Thus it was expected that this was the case also in
positive characteristic. Blickle, Musta\c{t}\u{a} and Smith proved
discreteness and rationality of the analogous positive characteristic
invariants in \cite{BMS}, but the proof was more involved.

In the mixed multiplier ideal setting, it follows from the above description
in terms of a log resolution that for every $b_{1},\ldots,b_{n}$
the region 
\[
\{(c_{1},\ldots,c_{n})\in\mathbb{R}_{\geq0}^{n}|\, c_{i}\leq b_{i}\mbox{ for all }i\}
\]
can be decomposed in a finite set of rational polytopes with nonoverlapping
interiors, such that on the interior of each face of each polytope
the mixed multiplier ideal $\mathcal{J}(\mathfrak{a}_{1}^{c_{1}}\cdots\mathfrak{a}_{n}^{c_{n}})$
is constant. It was expected that in the positive characteristic setting
we would have a similar picture.

In the present note we prove that this is not the case, but we can
still get a nice decomposition. This decomposition depends on a $p$-fractal
function, that is, a function $\varphi:\mathbb{R}_{\geq0}^{n}\rightarrow\mathbb{N}$
satisfying the following property. If we restrict $\varphi$ to a
bounded domain $D$, then the vector space generated by the functions
$\phi(t_{1},\ldots,t_{n})=\varphi((t_{1}+b_{1})/p^{e},\ldots,(t_{n}+b_{n})/p^{e})$
with $b_{i}$ integers and $((t_{1}+b_{1})/p^{e},\ldots,(t_{n}+b_{n})/p^{e})\in D$
, is finite dimensional (Definition \ref{Def Fractal function}).
Explicitly, we show:
\begin{thm*}
$[$Theorem \ref{thm:Main Theorem}$]$ For an $F$-finite, regular
ring $R$ essentially of finite type over a finite field of positive characteristic
and non zero ideals $\mathfrak{a}_{1},\ldots,\mathfrak{a}_{n}$ of
$R$, there is a $p$-fractal function $\varphi:\mathbb{R}_{\geq0}^{n}\rightarrow\mathbb{N}$
such that
\[
\tau(\mathfrak{a}_{1}^{c_{1}}...\mathfrak{a}_{n}^{c_{n}})=\tau(\mathfrak{a}_{1}^{d_{1}}...\mathfrak{a}_{n}^{d_{n}})\Longleftrightarrow\varphi(c_{1},\ldots,c_{n})=\varphi(d_{1},\ldots,d_{n}),
\]
 and therefore the constancy regions are of the form $\varphi^{-1}(i)$
for $i\in\mathbb{N}$.
\end{thm*}
Roughly speaking, this shows that each constancy region has a $p$-fractal
structure that, as we see in the examples in Section 5, can be intricate. 

This note is structured as follows. In section 2 we recall the definition
of test ideals and mixed test ideals following \cite{BMS} and state
some of the theorems that were proved there. In section 3 we give
our main definitions and deduce some basic consequences of these definitions.
We prove our main theorem in section 4. In the last section, we give
an example of a constancy region that is not a finite union of polyhedral
regions.

\section*{Acknowledgments}

I would like to thank Mircea Musta\c{t}\u{a} for introducing me to
this problem and for guiding my work. I would also like to thank Angélica
Benito, Luis Nuñez and Axel St\"{a}bler for comments and suggestions on earlier drafts.
Finally, I would like to thank Rui Huang for doing the graphs in section
5.

\section{Preliminaries}

Recall that a ring $R$ of positive characteristic is $\emph{F-finite}$
if the Frobenius morphism $F:R\rightarrow R$ is finite. Throughout
this note we let $R$ be a regular ring essentially of finite type
over an $F$-finite field $k$ of positive characteristic $p$. In
particular, $R$ is $F$-finite as well.We now recall the basic definitions
and properties related to test ideals and refer to \cite{BMS} for
proofs and details.

Given an ideal $\mf[b]$ in $R$, we denote by $\mathfrak{b}^{[1/p^{e}]}$
the smallest ideal $\mathfrak{J}$ such that $\mathfrak{b}\subseteq\mathfrak{J}^{[p^{e}]}:=(f^{p^{e}}|f\in\mathfrak{J})$.
The existence of a smallest such ideal is a consequence of the flatness
of the Frobenius map in the regular case. The following proposition
gives an explicit description of $\mf[b]^{[1/p^{e}]}$ when $R$ is
free over $R^{p^{e}}$.
\begin{prop}
\label{thm: root ideals}\cite[Proposition 2.5]{BMS} Suppose that
$R$ is free over $R^{q}$, for $q=p^{e}$, and let $e_{1},\ldots,e_{N}$
be a basis of $R$ over $R^{q}$. If $h_{1},\ldots,h_{n}$ are generators
of an ideal $\mathfrak{b}$ of $R$, and if for every $i=1,\ldots,n$
we write 
\[
h_{i}=\sum_{j=1}^{N}a_{i,j}^{q}e_{j}
\]
with $a_{i,j}\in R$, then 
\[
\mathfrak{b}^{[1/p^{e}]}=(a_{i,j}|i\leq n\mbox{ and }j\leq N).
\]

\end{prop}
Test ideals were introduced by Hochster and Huneke \cite{HH} as a
tool in their tight closure theory , and were later generalized by
Hara and Yoshida \cite{HY} in the context of pairs $(R,\mathfrak{a}^{c})$,
where $\mathfrak{a}$ is an ideal in $R$ and $c$ is a real parameter.
Blickle, Musta\c{t}\u{a}, and Smith \cite{BMS} gave an elementary
description of these ideals in the case of a regular $F$-finite ring
$R$ . It is this description which we take as our definition.
\begin{defn}
Given a non-negative number $c$ and a nonzero ideal $\mathfrak{a}$,
we define the $\textit{generalized}$ $\textit{test ideal}$ $\textit{of \ensuremath{\mf[a]\:}with exponent}$
$c$ to be 
\[
\tau(\mathfrak{a}^{c})=\bigcup_{e>0}(\mathfrak{a}^{\lceil cp^{e}\rceil})^{[1/p^{e}]},
\]
where $\lceil c\rceil$ stands for the smallest integer $\geq c.$ 
\end{defn}
The ideals in the above union form an increasing chain of ideals;
therefore as $R$ is Noetherian, they stabilize. Hence for $e$ large
enough $\tau(\mathfrak{a}^{c})=(\mathfrak{a}^{\lceil cp^{e}\rceil})^{[1/p^{e}]}.$
In the principal ideal case we can say more.
\begin{prop}
\label{Computing test ideals for principal ideals}\cite[Lemma 2.1]{BMS2}
If $\lambda=\frac{m}{p^{e}}$ for some positive integer $m$, then
$\tau(f^{\lambda})=(f^{m})^{[1/p^{e}]}.$
\end{prop}
It can be shown that as the parameter $c$ varies over the reals,
only countably many different test ideals appear; moreover, we have:
\begin{thm}
\cite[Proposition 2.14]{BMS} For every nonzero ideal $\mf[a]$ and
every non-negative number $c$, there exists $\epsilon>0$ such that
$\tau(\mf[a]^{c})=\tau(\mf[a]^{c'})$ for $c<c'<c+\epsilon$.\end{thm}
\begin{defn}
A positive real number $c$ is an $\textit{F-jumping exponent}$ of
$\mf[a]$ if $\tau(\mf[a]^{c})\neq\tau(\mf[a]^{c-\epsilon})$ for
all $\epsilon>0$
\end{defn}
The $F$-jumping exponents of an ideal $\mf[a]$ form a discrete set
of rational numbers, that is, there are no accumulation points of
this set. In fact, they form a sequence with limit infinity (see \cite[Theorem 3.1]{BMS}).

As in the case of one ideal, one can define the mixed test ideal of
several ideals as follows. 
\begin{defn}
\label{Def mixed test ideals}Given nonzero ideals $\mf[a]_{1},...,\mf[a]_{n}$
of $R$ and non-negative real numbers $c_{1},...,c_{n}$, we define
the $\textit{mixed generalized test ideal with exponents }c_{1},\ldots,c_{n}$
as:
\[
\tau(\mathfrak{a}_{1}^{c_{1}}\cdots a_{n}^{c_{n}})=\bigcup_{e>0}(\mathfrak{a}_{1}^{\lceil c_{1}p^{e}\rceil}\cdots\mathfrak{a}_{n}^{\lceil c_{n}p^{e}\rceil})^{[1/p^{e}]}.
\]

\end{defn}
As in the case of $\tau(\mathfrak{a}^{c})$, we have $\tau(\mathfrak{a}_{1}^{c_{1}}\cdots\mf[a]_{n}^{c_{n}})=(\mathfrak{a}_{1}^{\lceil c_{1}p^{e}\rceil}\cdots\mathfrak{a}_{n}^{\lceil c_{n}p^{e}\rceil})^{[1/p^{e}]}$
for all $e$ large enough. 
\begin{thm}
\label{thm:bound on the degree generators mixed case} Let $\mf[a]_{1},\ldots,\mf[a]_{n}$
be nonzero ideals in the polynomial ring $R=k[x_{1},\ldots,x_{r}]$,
and let $c_{1}=r_{1}/p^{s},\ldots,c_{n}=r_{n}/p^{s}$ be such that
$r_{1},\ldots,r_{n}$ are natural numbers. If each $\mf[a]_{i}$ can
be generated by polynomials of degree at most $d,$ then the ideal\textup{
$\tau(\mathfrak{a}_{1}^{c_{1}}\cdots\mf[a]_{n}^{c_{n}})$ can be generated
by polynomials of degree at most $\lfloor d(c_{1}+\ldots+c_{n})\rfloor$.
Here $\lfloor r\rfloor$ stands for the biggest integer $\leq r$.}\end{thm}
\begin{proof}
We argue as in \cite[Proposition 3.2]{BMS}, where the result was
proven for the case of one ideal. We know that $R$ is free over $R^{p^{e}}$
with basis 
\[
\{\beta_{j}x_{1}^{\alpha_{1}}\cdots x_{r}^{\alpha_{r}}|0\leq\alpha_{i}<p^{e}\mbox{ and \ensuremath{\beta_{j}}part of a basis for \ensuremath{k}over \ensuremath{k^{p^{e}}}}\}.
\]
The ideal $\mf[a]_{1}^{\lceil p^{e}c_{1}\rceil}\cdots\mf[a]_{n}^{\lceil p^{e}c_{n}\rceil}$
can be generated by polynomials of degree at most $d\lceil p^{e}c_{1}\rceil+\ldots+d\lceil p^{e}c_{n}\rceil$.
Hence taking $e>s$ large enough by Proposition \ref{thm: root ideals}
the ideal 
\[
\tau(\mathfrak{a}_{1}^{c_{1}}\cdots\mf[a]_{n}^{c_{n}})=(\mathfrak{a}_{1}^{\lceil p^{e}c_{1}\rceil}\cdots\mathfrak{a}_{n}^{\lceil p^{e}c_{n}\rceil})^{[1/p^{e}]}
\]
 is generated by polynomials of degree at most $(d\lceil p^{e}c_{1}\rceil+\ldots+d\lceil p^{e}c_{n}\rceil)/p^{e}=(dp^{e-s}r_{1}+\ldots+dp^{e-s}r_{n})/p^{e}=d(r_{1}+...+r_{n})$. 
\end{proof}

\section{Some Sets Associated to mixed test ideals}

In this section we introduce the definitions needed for our study
of mixed test ideals, and derive some basic properties. Recall that
$R$ denotes a regular ring essentially of finite type over an $F$-finite
field $k$ of positive characteristic. 
\begin{rem}
In order to simplify notation we denote $\mf[a]_{1}^{c_{1}}...\mf[a]_{n}^{c_{n}}$
by $\mf[a]^{\bs[c]}$, where $\mf[a]=(\mf[a]_{1},...,\mf[a]_{n})$,
$\bs[c]=(c_{1},...,c_{n})\in\mathbb{R}_{\geq0}^{n}$. We similarly
denote the vector $(\lceil r_{1}\rceil,...,\lceil r_{n}\rceil)$ by
$\lceil\bs[r]\rceil$, where $\bs[r]=(r_{1},\ldots,r_{n})\in\mathbb{R}_{\geq0}^{n}$.\end{rem}
\begin{defn}
\label{Def: Approximation sets}Given nonzero ideals $\mf[a]_{1},\ldots,\mf[a]_{n},$
and $I$ in $R$, we define
\[
V^{I}(\mf[a],p^{e})=\left\{ \frac{1}{p^{e}}\bs[c]=\left(\frac{c_{1}}{p^{e}},\ldots,\frac{c_{n}}{p^{e}}\right)\in\frac{1}{p^{e}}\mathbb{Z}_{\geq0}^{n}|\;\mf[a]^{\bs[c]}\not\subseteq I^{[p^{e}]}\right\} 
\]

and 
\[
B^{I}(\mf[a],p^{e})=\bigcup[0,l_{1}]\times\ldots.\times[0,l_{n}]\subset\mathbb{R}^{n},
\]
 where the union runs over all $(l_{1},\ldots,l_{n})\in V^{I}(\mf[a],p^{e})$.
\end{defn}
From this definition it follows that if $e'\geq e$ then $V^{I}(\mf[a],p^{e})\subseteq V^{I}(\mf[a],p^{e'})$
and $B^{I}(\mf[a],p^{e})\subseteq B^{I}(\mf[a],p^{e'}).$ Indeed,
if $\mf[a]^{\bs[c]}\not\subseteq I^{[p^{e}]}$, then there is an element
$f\in\mf[a]^{\bs[c]}$ with $f\notin I^{[p^{e}]}$, and by the flatness
of the Frobenius morphism we get $f^{p^{e'-e}}\in\mathfrak{a}^{p^{e'-e}\bs[c]}$
but $f^{p^{e'-e}}\notin I^{[p^{e'}]}$. Therefore $\mf[a]^{p^{e'-e}\bs[c]}\not\subseteq I^{[p^{e'}]}$,
hence we get the first inclusion. The second one is then straightforward. 
\begin{defn}
\label{Def characteristic function}Let $B^{I}(\mf[a])=\bigcup_{e>0}B^{I}(\mf[a],p^{e})$
and define $\chi_{\mf[a]}^{I}:\mathbb{R}^{n}\rightarrow\mathbb{N}$
to be the characteristic function of the set $B^{I}(\mf[a])$. That
is, $\chi_{\mf[a]}^{I}(\boldsymbol{c})$ is $1$ if $\boldsymbol{c}$
is in $B^{I}(\mf[a])$ and it is $0$ otherwise. 
\end{defn}
In order to study the sets $B^{I}(\mf[a])$ it is crucial to understand
how they intersect any increasing path. This motivates the following
definition.
\begin{defn}
Let $\mf[a]_{1},\ldots,\mf[a]_{n},$ and $I\not=R$ be nonzero ideals
as before and let $\bs[r]=(r_{1},\ldots,r_{n})\in\mathbb{Z}_{\geq0}^{n}$
be such that that $\mf[a]^{\bs[r]}\subseteq\mbox{rad}(I)$. We denote
\[
V_{\bs[r]}^{I}(\mf[a],p^{e})=\mbox{max}\{m\in\mathbb{Z}_{\geq0}|\mbox{ }\mf[a]^{m\bs[r]}\not\subseteq I^{[p^{e}]}\}.
\]
\end{defn}
\begin{rem}
While in the definition of $V^{I}(\mf[a],p^{e})$ one does not require
any relation between $\mf[a]$ and $I$, observe that we require that
$\mf[a]^{\bs[r]}\subseteq\mbox{rad}(I)$ when we consider $V_{\bs[r]}^{I}(\mf[a],p^{e})$.
\end{rem}
Note that if $\mf[a]^{m\bs[r]}\not\subseteq I^{[p^{e}]}$ then $a^{pm\bs[r]}\not\subseteq I^{[p^{e+1}]}$.
Therefore $pV_{\bs[r]}^{I}(\mf[a],p^{e})\le V_{\bs[r]}^{I}(\mf[a],p^{e+1})$,
hence
\begin{equation}
\left(\frac{V_{\bs[r]}^{I}(\mf[a],p^{e})}{p^{e}}\right)_{e\geq1}\label{Decreasing sequence}
\end{equation}
is a non-decreasing sequence.
\begin{prop}
The sequence \ref{Decreasing sequence} is bounded, hence it has a
limit.\end{prop}
\begin{proof}
If $\mf[a]^{\bs[r]}$ is generated by $s$ elements, then $\mf[a]^{(s(p^{e}-1)+1)\bs[r]}\subseteq(\mf[a]^{\bs[r]})^{[p^{e}]}.$
For $l$ large enough such that $\mf[a]^{l\bs[r]}\subseteq I$, we
have $V_{\bs[r]}^{I}(\mf[a],p^{e})\leq l(s(p^{e}-1)+1)-1$ for all
$e$. Therefore $V_{\bs[r]}^{I}(\mf[a],p^{e})/p^{e}\leq ls$, thus
the sequence is bounded.\end{proof}
\begin{defn}
We call this limit the\textit{ $F$-threshold of $\mathfrak{a}$ associated
to $I$ in direction $\boldsymbol{r}=(r_{1},\ldots,r_{n})$}, and
we denote it by $C_{\bs[r]}^{I}(\mf[a])$.\end{defn}
\begin{rem}
In the case $n=1$ we recover the usual definition of $F$-threshold
\cite{MTW}, \cite[Section 2.5]{BMS}.\end{rem}
\begin{lem}
\label{Contaiment in the positive direction}Let $\frac{1}{p^{e}}\bs[b]=(\frac{b_{1}}{p^{e}},\ldots,\frac{b_{1}}{p^{e}})$
and $\frac{1}{p^{e'}}\bs[c]=(\frac{c_{1}}{p^{e'}},\ldots,\frac{c_{n}}{p^{e'}})$
be two elements in \textup{$\mathbb{R}_{\geq0}^{n}$. If $\frac{b_{i}}{p^{e}}\leq\frac{c_{i}}{p^{e'}}$,
for every $i$, and $e'\leq e$ then $(\mf[a]^{\bs[c]})^{[1/p^{e'}]}\subseteq(\mf[a]^{\bs[b]})^{[1/p^{e}]}$.}\end{lem}
\begin{proof}
It follows as in \cite[Lemma 2.8]{BMS}. The condition $b_{i}\leq c_{i}p^{e-e'}$
implies that $\mf[a]_{i}^{b_{i}}\supseteq\mf[a]_{i}^{c_{i}p^{e-e'}}$for
every $i$. Therefore 
\[
(\mf[a]^{\bs[b]})^{[1/p^{e}]}\supseteq(\mf[a]^{p^{e-e'}\bs[c]})^{[1/p^{e}]}\supseteq(\mf[a]^{\bs[c]})^{[1/p^{e'}]}.
\]
 \end{proof}
\begin{prop}
\label{Constancy in positive directions}Given any $\bs[c]=(c_{1},\ldots,c_{n})\in\mathbb{R}_{\geq0}^{n}$,
there is $\bs[\epsilon]=(\epsilon_{1},\ldots,\epsilon_{n})\in\mathbb{R}_{>0}^{n}$
such that for every $\bs[r]=(r_{1},\ldots,r_{n})$ with $0<r_{i}<\epsilon_{i}$,
we have $\tau(\mf[a]^{\bs[c]})=\tau(\mf[a]^{\bs[c]+\bs[r]})$.\end{prop}
\begin{proof}
We argue as in the proof of \cite[Proposition 2.14]{BMS}. We first
show that there is a vector $\bs[\epsilon]=(\epsilon_{1},\ldots,\epsilon_{n})$,
with $\epsilon_{i}>0$ for all $i$, such that for all vectors $\bs[r]=(r_{1},\ldots,r_{n})\in\mathbb{Z}^{n}$
with $c_{i}<\frac{1}{p^{e}}r_{i}<c_{i}+\epsilon_{i}$ we have that
$(\mf[a]^{\bs[r]})^{[1/p^{e}]}$ is constant. Indeed, otherwise there
are sequences $\bs[r]_{m}=(r_{m,1},\ldots,r_{m,n})\in\mathbb{Z}_{\geq0}^{n}$
and $e_{m}\in\mathbb{Z}_{\geq0}$ such that $\frac{1}{p^{e_{m}}}\bs[r]_{m}$
converges to $\bs[c]$, $\left(\frac{1}{p^{e_{m}}}r_{m,i}\right)_{m}$
is a decreasing sequence for every $i$, $e_{m}\leq e_{m+1}$, and
$(\mf[a]^{\bs[r]_{m}})^{[1/p^{e_{m}}]}\not=(\mf[a]^{\bs[r]_{m+1}})^{[1/p^{e_{m+1}}]}$.
It follows from Lemma \ref{Contaiment in the positive direction}
that $(\mf[a]^{\bs[r]_{m}})^{[1/p^{e_{m}}]}\subsetneq(\mf[a]^{\bs[r]_{m+1}})^{[1/p^{e_{m+1}}]}$
for all $m$, but this contradicts the fact that $R$ is Noetherian. 

Assume now that $\bs[\epsilon]=(\epsilon_{1},\ldots,\epsilon_{n})$
is as above and let $I=(\mf[a]^{\bs[r]})^{[1/p^{e}]}$ for all $\bs[r]=(r_{1},...,r_{n})\in\mathbb{Z}^{n}$
with $c_{i}<\frac{1}{p^{e}}r_{i}<c_{i}+\epsilon_{i}$. We show that
$I=\tau(\mf[a]^{\bs[c]})$. Take $e$ large enough such that $\tau(\mf[a]^{\bs[c]})=(\mf[a]^{\lceil p^{e}\bs[c]\rceil})^{[1/p^{e}]}$
and $\frac{\lceil p^{e}c_{i}\rceil}{p^{e}}<c_{i}+\epsilon_{i}$ for
every $i$. If all $p^{e}c_{i}$ are non-integers then $\frac{\lceil p^{e}c_{i}\rceil}{p^{e}}>c_{i}$
and $\tau(\mf[a]^{\bs[c]})=I$. Let us suppose that $p^{e}c_{i}$
is an integer precisely when $i=i_{1},...,i_{l}$. Let $\bs[d]=(d_{1},\ldots,d_{n})$
be the vector whose $i_{j}$ coordinates are $1$ and all the other
are $0$. As $e$ is arbitrarily large we may also assume that $c_{i}<c_{i}+\frac{1}{p^{e}}d_{i}<c_{i}+\epsilon_{i}$
for all $i\in\{i_{1},\ldots,i_{l}\}$, hence $I=(\mf[a]^{\lceil p^{e}\bs[c]\rceil+\boldsymbol{d}})^{[1/p^{e}]}\subseteq(\mf[a]^{\lceil p^{e}\bs[c]\rceil})^{[1/p^{e}]}=\tau(\mf[a]^{\bs[c]})$. 

The reverse inclusion follows by showing $\mf[a]^{\lceil p^{e}\bs[c]\rceil}\subseteq I^{[p^{e}]}$.
Let $u\in\mf[a]^{\lceil p^{e}\bs[c]\rceil}$. If $e'>e$ and $e'$
is large enough, then $c_{i}<c_{i}+\frac{1}{p^{e'}}<c_{i}+\epsilon_{i}$,
hence $\mf[a]^{\lceil p^{e'}\bs[c]\rceil+\boldsymbol{1}}\subseteq I^{[p^{e'}]}$.
Here $\boldsymbol{1}$ denotes the vector whose coordinates are all
$1$. Thus, for $v$ a nonzero element in $\mf[a]_{1}\cdots\mf[a]_{n}$
we have 
\[
vu^{p^{e'-e}}\in\mbox{\ensuremath{\mf[a]}}^{p^{e'-e}\lceil p^{e}\bs[c]\rceil+\boldsymbol{1}}\subseteq\mf[a]^{\lceil p^{e'}\bs[c]\rceil+\boldsymbol{1}}\subseteq(I^{[p^{e}]})^{[p^{e'-e}]}.
\]
This implies that $u$ is in the tight closure of $I^{[p^{e}]}$,
but as $R$ is a regular ring, the tight closure of $I^{[p^{e}]}$
is equal to $I^{[p^{e}]}$(see \cite{HH}). This gives $\mathfrak{a}^{\lceil p^{e}\boldsymbol{c}\rceil}\subseteq I^{[p^{e}]}$
hence, by definition, $\tau(\mathfrak{a}^{\boldsymbol{c}})=(\mathfrak{a}^{\lceil p^{e}\boldsymbol{c}\rceil})^{[1/p^{e}]}\subseteq I$. \end{proof}
\begin{defn}
A positive real number $c$ is called an \textit{$F$-jumping number
of $\mf[a]$ in the direction $\bs[r]\not=\bs[0]\in\mathbb{Z}_{\geq0}^{n}$},
if $c$ is such that $\tau(\mf[a]^{c\bs[r]})\not=\tau(\mf[a]^{(c-\epsilon)\bs[r]})$
for every real number $\epsilon>0$. \end{defn}
\begin{prop}
\label{Mixed test ideals =00003D normal test ideals}If $r\in\mathbb{Z}_{\geq0}^{n}$
and $\lambda\in\mathbb{R}_{\geq0}$, then 
\[
\tau(\mf[a]^{\lambda r_{1}}\cdots\mf[a]^{\lambda r_{n}})=\tau(\mf[J]^{\lambda}),
\]

where $\mf[J]=\mf[a]_{1}^{r_{1}}\cdots\mf[a]_{n}^{r_{n}}$.\end{prop}
\begin{proof}
By Propostion \ref{Constancy in positive directions}, we may assume
$\lambda=\frac{s}{p^{e'}}$ with $s\in\mathbb{Z}_{\geq0}$. For $e$
sufficiently large, we have 
\[
\tau(\mf[a]^{\lambda r_{1}}\cdots\mf[a]^{\lambda r_{n}})=(\mf[a]^{\lceil\lambda r_{1}p^{e}\rceil}\cdots\mf[a]^{\lceil\lambda r_{n}p^{e}\rceil})^{[1/p^{e}]}=(\mf[a]^{sr_{1}p^{e-e'}}\cdots\mf[a]^{sr_{n}p^{e-e'}})^{[1/p^{e}]}
\]
\[
=((\mf[a]^{r_{1}}\cdots\mf[a]^{r_{n}})^{sp^{e-e'}})^{[1/p^{e}]}=((\mf[a]^{r_{1}}\cdots\mf[a]^{r_{n}})^{\lambda p^{e}})^{[1/p^{e}]}=\tau(\mf[J]^{\lambda}).
\]
\end{proof}
\begin{cor}
The $F$-threshold of $\mathfrak{a}$ associated to $I$ in the direction
$\boldsymbol{r}=(r_{1},\ldots,r_{n})$ is equal to the $F$-threshold
of $\mf[a]_{1}^{r_{1}}\cdots\mf[a]_{n}^{r_{n}}$ associated to $I$. 
\end{cor}
\medskip{}

\begin{cor}
The set of $F$ -jumping numbers of $\mf[a]$ in direction ${\bf r}$
is equal to the set of $F$-jumping numbers of $\mf[a]^{{\bf r}}$. 
\end{cor}
Therefore \cite[Corollary 2.30]{BMS} implies the following. 
\begin{cor}
The set of $F$-jumping numbers of $\mf[a]$ in the direction $\bs[r]$
is equal to the set of $F$-thresholds of $\mf[a]$, associated to
various ideals $I$, in the direction $\bs[r]$.
\end{cor}
Given $l_{1},\ldots,l_{n}$ positive real numbers we denote by $[\boldsymbol{0},\boldsymbol{l}]$
the set $[0,l_{1}]\times\ldots\times[0,l_{n}].$
\begin{prop}
\label{Finite many test ideals} Given nonzero ideals $\mf[a]_{1},\ldots,\mf[a]_{n}$
of $R$, where $R$ is a regular, $F$-finite ring essentially of finite type over a finite field, the set $\{\tau(\mf[a]^{\bs[c]})|\:\bs[c]\in[\bs[0],\bs[l]]\}$
is finite.\end{prop}
\begin{proof}
Since $R$ is assumed to be essentially of finite type over $k$,
arguing as in the proof of \cite[Theorem 3.1]{BMS}, one can see that
the assertion for all such $R$ follows if we know it for $R=k[x_{1},\ldots x_{r}]$,
with $r\geq1.$ We will therefore assume that we are in this case.

By Lemma \ref{Constancy in positive directions}, we may assume that
$\bs[c]=(\frac{\alpha_{1}}{p^{e}},...,\frac{\alpha_{n}}{p^{e}})$
with $\alpha_{i}\in\mathbb{N}$ and $e\geq1$. Let $d$ be an upper
bound for the degrees of the generators of $\mf[a_{i}]$, for all
$i$. By Theorem \ref{thm:bound on the degree generators mixed case}
we have that $\tau(\mf[a]^{\bs[c]})$ is generated by polynomials
of degree $\leq ndL$, where $L=\mbox{max}\{l_{i}\}$. Since $k$
is finite, there are only finitely many sets consisting of polynomials
of bounded degree and therefore only finitely many ideals $\tau(\mf[a]^{\bs[c]})$
where $\bs[c]\in[\bs[0],\bs[l]]$. 

\end{proof}

\begin{defn}
The $\textit{constancy region}$ for a test ideal $\tau(\mf[a]^{\boldsymbol{c}})$
is defined as the set of points $\boldsymbol{c'}\in\mathbb{R}_{\geq0}^{n}$
such that $\tau(\mf[a]^{\boldsymbol{c}})=\tau(\mf[a]^{\boldsymbol{c}'})$. \end{defn}
\begin{lem}
\textup{\label{lem:Description of the approximation}$B^{J}(\mf[a])$
consist of the points ${\bf c}\in\mathbb{R}_{\geq0}^{n}$ such that
$\tau(\mf[a]^{{\bf c}})\not\subseteq J.$}\end{lem}
\begin{proof}
Assume first that ${\bf c}=(\frac{\alpha_{1}}{p^{e}},\ldots,\frac{\alpha_{n}}{p^{e}})$
with $\alpha_{i}\in\mathbb{N}$. Choose a representation of ${\bf c}$
with $e$ large enough such that $\tau(\mbox{\ensuremath{\mf[a]}}^{{\bf c}})=(\mbox{\ensuremath{\mf[a]}}^{{\bf {\bf \alpha}}})^{[1/p^{e}]}.$
In this case we have 

\[
{\bf c}\in B^{J}(\mf[a])\Longleftrightarrow\mbox{\ensuremath{\mf[a]}}^{{\bf \alpha}}\not\subseteq J^{[p^{e}]}\Longleftrightarrow(\mbox{\ensuremath{\mf[a]}}^{{\bf \alpha}})^{[1/p^{e}]}\not\subseteq J\Longleftrightarrow\tau(\mbox{\ensuremath{\mf[a]}}^{{\bf c}})\not\subseteq J.
\]

For the general case, let ${\bf c}\in B^{J}(\mf[a])$, this implies
that ${\bf c}\in B^{J}(\mf[a],p^{e})$ for some $e$. Therefore we
can find ${\bf r}=(\frac{\alpha_{1}}{p^{e}},\ldots,\frac{\alpha_{n}}{p^{e}})\in B^{J}(\mf[a],p^{e})\subseteq B^{J}(\mf[a])$,
with $\alpha_{i}\in\mathbb{N}$, $\frac{\alpha_{i}}{p^{e}}\geq c_{i}$.
By the first part this implies $\tau(\mbox{\ensuremath{\mf[a]}}^{{\bf r}})\not\subseteq J$,
but as $\frac{\alpha_{i}}{p^{e}}\geq c_{i}$ for all $i$, we have
that $\tau(\mbox{\ensuremath{\mf[a]}}^{{\bf r}})\subseteq\tau(\mbox{\ensuremath{\mf[a]}}^{{\bf c}})$
hence $\tau(\mbox{\ensuremath{\mf[a]}}^{{\bf c}})\not\subseteq J$.

For the reverse inclusion, let ${\bf c}\in\mathbb{R}_{\geq0}^{n}$
be such that $\tau(\mf[a]^{{\bf c}})\not\subseteq J.$ By Proposition
\ref{Constancy in positive directions} there is a point ${\bf r}=(\frac{\alpha_{1}}{p^{e}},\ldots,\frac{\alpha_{n}}{p^{e}})$
with $\alpha_{i}\in\mathbb{N}$, $\frac{\alpha_{i}}{p^{e}}\geq c_{i}$
and $\tau(\mf[a]^{{\bf c}})=\tau(\mf[a]^{{\bf r}})$, therefore $\tau(\mf[a]^{{\bf r}})\not\subseteq J$.
We use the first part again and conclude ${\bf r}\in B^{J}(\mf[a])$,
but as $\frac{\alpha_{i}}{p^{e}}\geq c_{i}$ for all $i$, we deduce
that ${\bf c}\in B^{J}(\mf[a])$.\end{proof}

\begin{thm}
If $\mf[a]_{1},...,\mbox{\ensuremath{\mf[a]}}_{n}$ are all contained
in a maximal ideal $\mf[m]$ and the base field $k$ is finite, then for each $\boldsymbol{c}\in\mathbb{R}_{\geq0}^{n}$,
there exist ideals $I_{1},...,I_{d}$ and $J$ such that the constancy
region for the test ideal $\tau(\mf[a]^{\bs[c]})$ is given by $\underset{i=1,...,d}{\bigcap}B^{I_{i}}(\mf[a])\backslash B^{J}(\mf[a])$.\end{thm}
\begin{proof}
We first show that this constancy region is bounded. As $\mf[a]_{i}\subseteq\mf[m]$
for all $i$ we have that for any $\boldsymbol{c'}\in\mathbb{R}_{\geq0}^{n}$and
$e$ sufficiently large 
\[
\tau(\mf[a]^{{\bf c'}})=(\mf[a]_{1}^{\lceil c'_{1}p^{e}\rceil}\cdots\mf[a]_{n}^{\lceil c'_{n}p^{e}\rceil})^{[1/p^{e}]}\subseteq(\mf[m]^{\lceil c'_{1}p^{e}\rceil+\ldots+\lceil c'_{n}p^{e}\rceil})^{[1/p^{e}]}
\]
\[
\subseteq\mbox{\ensuremath{\mf[(]}}\mf[m]^{\lceil c'_{1}p^{e}+\ldots+c'_{n}p^{e}\rceil-n})^{[1/p^{e}]}\subseteq\mf[m]^{\lceil c'_{1}+\ldots+c'_{n}\rceil-n+1}.
\]

Since $\cap_{s}\mf[m]^{s}=0$, there is $L$ such that $\tau(\mf[a]^{{\bf c}})\not\subseteq\mbox{\ensuremath{\mf[m]}}^{L}$,
we deduce that for any ${\bf c}'$ in the constancy region $\tau(\mf[a]^{{\bf c}'})=\tau(\mf[a]^{{\bf c}})\not\subseteq\mbox{\ensuremath{\mf[m]}}^{L}$,
hence $c'_{1}+\ldots+c'_{n}\leq L$. This implies that the constancy
region for $\tau(\mf[a]^{\bs[c]})$ is bounded. 

To deduce our description consider a sufficiently large hypercube
$[\bs[0],\bs[l]]$ containing the constancy region for $\tau(\mf[a]^{{\bf c}})$.
By Proposition \ref{Finite many test ideals}, we know that the set
$\mathcal{A}=\{\tau(\mf[a]^{\bs[c]})|\:\bs[c]\in[\bs[0],\bs[l]]\}$
is finite. Let $I_{1},\ldots,I_{d}$ be the ideals in $\mathcal{A}$
that are strictly contained in $\tau(\mf[a]^{{\bf c}})$ and let $J=\tau(\mf[a]^{{\bf c}})$.
We claim that the constancy region for $\tau(\mf[a]^{{\bf c}})$ is
equal to $\underset{i=1,...,d}{\bigcap}B^{I_{i}}(\mf[a])\backslash B^{J}(\mf[a])$.
Lemma \ref{lem:Description of the approximation} implies that the
set $\underset{i=1,...,d}{\bigcap}B^{I_{i}}(\mf[a])\backslash B^{J}(\mf[a])$
is equal to the set of all ${\bf r}$ such that $\tau(\mf[a]^{{\bf r}})\not\subseteq I_{i}$
for all $i$ and $\tau(\mf[a]^{{\bf c}})\subseteq\tau(\mf[a]^{{\bf r}})$,
or equivalently, $\tau(\mf[a]^{{\bf r}})=\tau(\mf[a]^{{\bf c}})$
by our choice of $I_{i}$.\end{proof}
\begin{rem}
\label{Regions}We can remove the condition that all ideals $\mf[a]_{i}$
are contained in a maximal ideal and still get a similar description.
Explicitly, in each hypercube $[\bs[0],\bs[l]]$ the constancy region
is given by $\underset{i=1,...,d}{\bigcap}\big(B^{I_{i}}(\mf[a])\backslash B^{J}(\mf[a])\big)\cap[\bs[0],\bs[l]],$
for suitable $I_{1},\ldots,I_{d}$ and $J$.
\end{rem}
We now give a version of Skoda's theorem for mixed test ideals (see
\cite[Proposition 2.25]{BMS} for the case of one ideal). This theorem
allows us to describe the constancy regions in the first octant by
describing only the constancy regions in a sufficiently large hypercube
$[\bs[0],\bs[l]]=[0,l_{1}]\times\ldots\times[0,l_{n}]$.
\begin{thm}
\label{thm:(Skoda's-Theorem)}(Skoda's Theorem) Let $\bs[e_{1}],...,\bs[e_{n}]$
be the standard basis for $\mathbb{R}^{n}$, and assume $1\leq i\leq n$.
If $\mf[a]_{i}$ is generated by $m_{i}$ elements, then for every
$\bs[s]=(s_{1},\ldots,s_{n})$ with $s_{i}\geq m_{i}$, we have 
\[
\tau(\mf[a]^{\bs[s]})=\mf[a]^{\bs[e_{i}]}\tau(\mf[a]^{\bs[s]-\bs[e_{i}]}).
\]
\end{thm}
\begin{proof}
We only need to prove $(\mf[a]^{\lceil p^{e}\bs[s]\rceil})^{[1/p^{e}]}=\mf[a]^{\bs[e_{i}]}(\mf[a]^{\lceil p^{e}(\bs[s]-\bs[e_{i}])\rceil})^{[1/p^{e}]}$
for $e$ large enough. 

Let $\bs[d]=(d_{1},..,d_{n})$ be a vector with integer coordinates
and $d_{i}\geq p^{e}s_{i}$. We want to show that 
\[
(\mf[a]^{\bs[d]})^{[1/p^{e}]}=\mf[a]^{\bs[e_{i}]}(\mf[a]^{\bs[d]-p^{e}\bs[e_{i}]})^{[1/p^{e}]},
\]
from which the result follows. 

Since $\mf[a]^{\bs[d]-p^{e}\bs[e_{i}]}\cdot\mf[a]_{i}^{[p^{e}]}\subseteq\mbox{\ensuremath{\mf[a]}}^{\bs[d]}\subseteq((\mf[a]^{\bs[d]})^{[1/p^{e}]})^{[p^{e}]}$,
we have
\[
\mf[a]^{\boldsymbol{d}-p^{e}\bs[e_{i}]}\subseteq\big(((\mf[a]^{\bs[d]})^{[1/p^{e}]})^{[p^{e}]}:\mf[a]_{i}^{[p^{e}]}\big)=\big((\mf[a]^{\bs[d]})^{[1/p^{e}]}:\mf[a]_{i}\big)^{[p^{e}]},
\]
where the equality is consequence of the flatness of Frobenius. Therefore
\[
(\mf[a]^{\bs[d]-p^{e}\bs[e_{i}]})^{[1/p^{e}]}\subseteq((\mf[a]^{\bs[d]})^{[1/p^{e}]}:\mf[a]_{i}),
\]
 that is, 
\[
\mf[a]^{\bs[e_{i}]}(\mf[a]^{\bs[d]-p^{e}\bs[e_{i}]})^{[1/p^{e}]}\subseteq(\mf[a]^{\bs[d]})^{[1/p^{e}]}.
\]
For the reverse inclusion, note that since $d_{i}\geq m_{i}(p^{e}-1)+1$,
in the product of $d_{i}$ of the generators of $\mf[a]_{i}$ at least
one should appear with multiplicity $\geq p^{e}$. Therefore $\mf[a]^{\bs[d]}=\mf[a]_{i}^{[p^{e}]}\cdot\mf[a]^{\bs[d]-p^{e}\bs[e_{i}]}$,
hence
\[
\mf[a]^{\bs[d]}\subseteq\mf[a]_{i}^{[p^{e}]}\cdot\mf[a]^{\bs[d]-p^{e}\bs[e_{i}]}\subseteq\mf[a]_{i}^{[p^{e}]}\cdot\big((\mf[a]^{\bs[d]-p^{e}\bs[e_{i}]})^{[1/p^{e}]}\big)^{[p^{e}]}=\big(\mf[a]^{\bs[e_{i}]}\cdot(\mf[a]^{\bs[d]-p^{e}\bs[e_{i}]})^{[1/p^{e}]}\big)^{[p^{e}]},
\]
which clearly implies $(\mf[a]^{\bs[d]})^{[1/p^{e}]}\subseteq\mf[a]^{\bs[e_{i}]}(\mf[a]^{\bs[d]-p^{e}\bs[e_{i}]})^{[1/p^{e}]}.$ \end{proof}
\begin{prop}
If $c$ is an $F$-jumping number in the direction $\bs[r]=(r_{1},\ldots,r_{n})$
then also $cp$ is an $F$-jumping number in the direction $\bs[r]$.\end{prop}
\begin{proof}
Note that $V_{\bs[r]}^{I}(\mf[a],p^{e+1})=V_{\bs[r]}^{I^{[p]}}(\mf[a],p^{e})$,
hence $pC_{\bs[r]}^{I}(\mf[a])=C_{\bs[r]}^{I^{[p]}}(\mf[a])$.
\end{proof}

\section{The constancy regions}

In this section we prove our main result, Theorem \ref{thm:Main Theorem}
below. We begin by recalling our definition of $p$-fractals.

Let $\mathcal{F}$ be the algebra of functions $\phi:\mathbb{R}_{\geq0}^{n}\rightarrow\mathbb{Q}$.
For each $q=p^{e}$ and every $\bs[b]=(b_{1},\ldots,b_{n})\in\mathbb{Z}^{n}$
with $0\leq b_{i}<q$ we define a family of operators $T_{q|\bs[b]}:\mathcal{F}\rightarrow\mathcal{F}$
by 
\[
T_{q|\bs[b]}\phi(t_{1},\ldots,t_{n})=\phi((t_{1}+b_{1})/q,\ldots,(t_{n}+b_{n})/q).
\]

\begin{defn}
\label{Def Fractal function}Let $\phi:[0,l]^{n}\rightarrow\mathbb{Q}$
be a map and let denote also by $\phi$ its extension by zero to $\mathbb{R}_{\geq0}^{n}$.
We say that $\phi$ is a $p\textit{-fractal}$ if all the $T_{q|\bs[b]}\phi$
span a finite dimensional $\mathbb{Q}$-subspace $V$ of $\mathcal{F}$.
Furthermore, we say that an arbitrary $\phi\in\mathcal{F}$ is a $p\textit{-fractal}$
if its restriction to each hypercube $[\bs[0],\bs[l]]$ is a $p$-fractal.\end{defn}
\begin{rem}
This definition is similar to the one in \cite[Definition 2.1]{MT}.
The only difference is that in \cite[Definition 2.1]{MT} the domain
of the functions is the hypercube $[0,1]\times\ldots\times[0,1].$
\end{rem}
In this section we assume that $R$ is a regular, $F$-finite ring
essentially of finite type over a finite field of characteristic $p>0$,
and $\mf[a]_{i}\subseteq R$ are nonzero ideals.
\begin{lem}
\label{lemma for the main theorem} Let $\bs[c]=(c_{1},\ldots,c_{n})\in\mathbb{R}_{\geq0}^{n}$,
and $\mf[a]_{1},\ldots\mf[a]_{n}$ be nonzero ideals of $R$ then
$\tau(\mf[a]^{\bs[c]})^{[1/p^{e}]}=\tau(\mf[a]^{\frac{1}{p^{e}}\bs[c]})$. \end{lem}
\begin{proof}
Taking $k$ large enough 
\[
\tau(\mf[a]^{{\bf c}})^{[1/p^{e}]}=\left((\mf[a]_{1}^{\lceil c{}_{1}p^{k}\rceil}\cdots\mf[a]_{n}^{\lceil c{}_{n}p^{k}\rceil})^{[1/p^{k}]}\right)^{[1/p^{e}]}
\]
and by \cite[Lemma 2.4]{BMS} the later contains 
\[
(\mf[a]_{1}^{\lceil c{}_{1}p^{k}\rceil}\cdots\mf[a]_{n}^{\lceil c{}_{n}p^{k}\rceil})^{[1/p^{k+e}]}=(\mf[a]_{1}^{\lceil\frac{c{}_{1}}{p^{e}}p^{k+e}\rceil}\cdots\mf[a]_{n}^{\lceil\frac{c{}_{n}}{p^{e}}p^{k+e}\rceil})^{[1/p^{k+e}]}
\]
\[
=\tau(\mf[a]^{\frac{1}{p^{e}}\bs[c]}).
\]

Therefore $\tau(\mf[a]^{{\bf c}})^{[1/p^{e}]}\supseteq\tau(\mf[a]^{\frac{1}{p^{e}}\bs[c]}).$

For the other inclusion note that 
\[
\tau(\mf[a]^{{\bf c}})=(\mf[a]_{1}^{\lceil c{}_{1}p^{k}\rceil}\cdots\mf[a]_{n}^{\lceil c{}_{n}p^{k}\rceil})^{[1/p^{k}]}
\]
\[
=(\mf[a]_{1}^{\lceil\frac{c{}_{1}}{p^{e}}p^{k+e}\rceil}\cdots\mf[a]_{n}^{\lceil\frac{c{}_{n}}{p^{e}}p^{k+e}\rceil})^{[p^{e}/p^{k+e}]}
\]
that by \cite[Lemma 2.4]{BMS} is contained in 
\[
\left((\mf[a]_{1}^{\lceil\frac{c{}_{1}}{p^{e}}p^{k+e}\rceil}\cdots\mf[a]_{n}^{\lceil\frac{c{}_{n}}{p^{e}}p^{k+e}\rceil})^{[1/p^{k+e}]}\right)^{[p^{e}]}=\tau(\mf[a]^{\frac{1}{p^{e}}\bs[c]})^{[p^{e}]}
\]
but this is equivalent to say 
\[
\tau(\mf[a]^{{\bf c}})^{[1/p^{e}]}\subseteq\tau(\mf[a]^{\frac{1}{p^{e}}\bs[c]}).
\]
\end{proof}
\begin{lem}
\label{characterisct functions fractal behavior} Let $\boldsymbol{l}=(l_{1},\ldots l_{n})\in\mathbb{Z}^{n}$
be such that $l_{i}$ is the minimum number of generators of the ideal
$\mf[a_{i}]$. Let $\bs[b]\in\mathbb{Z}^{n}$ such that $l_{i}-1\leq b_{i}$
. For all $e$, we have \textup{
\[
T_{p^{e}|\bs[b]}\chi_{\mf[a]}^{I}=T_{p^{0}|(\bs[l]-\bs[1])}\chi_{\mf[a]}^{(I^{[p^{e}]}:\mf[a]^{\bs[b-l+1]})},
\]
}where $\chi_{\mf[a]}^{I}$ denotes the characteristic function introduced
in Definition \ref{Def characteristic function}.\end{lem}
\begin{proof}
We have that 
\[
T_{p^{e}|\bs[b]}\chi_{\mf[a]}^{I}(\bs[t])=\chi_{\mf[a]}^{I}\left(\frac{1}{p^{e}}(\bs[t]+\bs[b])\right)
\]

is equal to $1$ if and only if, by Lemma \ref{lem:Description of the approximation},
to 
\[
\tau(\mf[a]^{\frac{1}{p^{e}}(\bs[t]+\bs[b])})\not\subseteq I,
\]

and by Lemma \ref{lemma for the main theorem} this is 
\[
\tau(\mf[a]^{\bs[t]+\bs[b]})^{[1/p^{e}]}\not\subseteq I,
\]

but the later is equivalent to 
\[
\tau(\mf[a]^{\bs[t]+\bs[b]})\not\subseteq I^{[p^{e}]}.
\]

As $b_{i}\geq l_{i}-1$ by Skoda's Theorem the previous expresion
becomes 
\[
\mf[a]^{\bs[b-l+1]}\cdot\tau(\mf[a]^{\bs[t]+\bs[l]-\bs[1]})\not\subseteq I^{[p^{e}]}
\]

Wich in turn is equivalent to 
\[
\tau(\mf[a]^{\bs[t]+\bs[l]-\bs[1]})\not\subseteq(I^{[p^{e}]}:\mf[a]^{\bs[b-l+1]})
\]
 but this is the case if and only if 
\[
T_{p^{0}|(\bs[l]-\bs[1])}\chi_{\mf[a]}^{(I^{[p^{e}]}:\mf[a]^{\bs[b-l+1]})}(\bs[t])=\chi_{\mf[a]}^{(I^{[p^{e}]}:\mf[a]^{\bs[b-l+1]})}(\bs[t]+\bs[l]-\bs[1])
\]

is equal to $1$ 

Note that a point of the form $\frac{1}{p^{k}}\bs[r]+(\boldsymbol{l}-\boldsymbol{1})$
with $\bs[r]\in\mathbb{Z}^{n}$ is in $B^{(I^{[p^{e}]}:\mf[a]^{b-l+1})}(\mf[a])$
if and only if $\mf[a]^{\bs[r]+p^{k}(\bs[l]-\bs[1])}\not\not\subseteq(I^{[p^{e}]}:\mf[a]^{\bs[b]-\bs[l]+\bs[1]})^{[p^{k}]}$
if and only if $\mf[a]^{\bs[r]}\cdot\mf[a]^{p^{k}(\bs[l]-\bs[1])}\cdot(\mf[a]^{\bs[b]-\bs[l]+\bs[1]})^{[p^{k}]}\not\subseteq I^{[p^{e+k}]}$
this by Lemma \ref{lemma for the main theorem} occurs if and only
if $\mf[a]^{\bs[r]+p^{k}\bs[b]}\not\subseteq I^{[p^{e+k}]}$, or equivalently
$\frac{1}{p^{e+k}}\bs[r]+\frac{1}{p^{e}}\bs[b]\in B^{I}(\mf[a])$.
From this the result follows easily. 
\end{proof}
This lemma is especially useful when the ideals are principal, as
we will see in the examples of Section 5. 
\begin{lem}
\label{lem:In-each-hypercube} In each hypercube $[\bs[0],\bs[l]]$
there are only finitely many functions $\chi_{\mf[a]}^{I}$. That
is, the set $\{\chi_{\mf[a]}^{I}|_{[\bs[0,l]]};I\subseteq R\}$ is
finite. \end{lem}
\begin{proof}
By Lemma \ref{lem:Description of the approximation}, $B^{I}(\mf[a])$
is the set of all points $\bs[c]=(c_{1},\ldots,c_{n})\in\mathbb{R}_{\geq0}^{n}$
such that $\tau(\mf[a]^{\bs[c]})\not\subseteq I$, hence $B^{I}(\mf[a])$
is a union of constancy regions. By Lemma \ref{Finite many test ideals},
we know that there are only finitely many constancy regions for bounded
exponents, therefore there are only finitely many functions $\chi_{\mf[a]}^{I}|_{[\bs[0],\bs[l]]}$. \end{proof}
\begin{thm}
\textup{\label{thm:Main Theorem}There is a $p$-fractal function
$\varphi:\mathbb{R}_{\geq0}^{n}\rightarrow\mathbb{N}$ for which 
\[
\tau(\mathfrak{a}_{1}^{c_{1}}...\mathfrak{a}_{n}^{c_{n}})=\tau(\mathfrak{a}_{1}^{d_{1}}...\mathfrak{a}_{n}^{d_{n}})\Longleftrightarrow\varphi(c_{1},...,c_{n})=\varphi(d_{1},...,d_{n}),
\]
}and therefore the constancy regions are of the form $\varphi^{-1}(i)$
for some number $i.$\end{thm}
\begin{proof}
We first show that the functions $\chi_{\mf[a]}^{I}$ are $p$-fractal.
We want to prove that all the $T_{p^{e}|\bs[b]}\chi_{\mf[a]}^{I}$
span a finite dimensional space. Lemma \ref{characterisct functions fractal behavior}
states that all but finitely many of these functions have the form
$T_{p^{0}|(\bs[l]-\bs[1])}\chi_{\mf[a]}^{J}$ for different ideals
$J$. Lemma \ref{lem:In-each-hypercube} ensures that there are only
finitely many of those in each hypercube $[\bs[0],\bs[l]]$. From
this it follows that $\chi_{\mf[a]}^{I}$ is a $p$-fractal.

For $\bs[c]\in\mathbb{R}_{\geq0}^{n}$, let $\eta_{\bs[c]}$ be the
characteristic function associated to the constancy region $\tau(\mf[a]^{\bs[c]})$.
Remark \ref{Regions} implies that that in each hypercube $[\bs[0],\bs[l]]$,
$\eta_{\bs[c]}|_{[\bs[0],\bs[l]]}=\left(\chi_{\mf[a]}^{I_{1}}\cdots\chi_{\mf[a]}^{I_{d}}-\chi_{\mf[a]}^{J}\right)\big|{}_{[\bs[0],\bs[l]]}$
for some ideals $I_{1},\ldots,I_{d}$ and $J$, therefore $\eta_{\bs[c]}$
is $p$-fractal. 

Clearly there are countably many constancy regions, so we can numerate
them. For every $i$, let $\bs[c]_{i}=(c_{i1},\ldots,c_{in})$ a point
in the $i$-th constancy region, and we define 
\[
\varphi=\sum_{i\in\mathbb{N}}i\cdot\eta_{\bs[c]_{i}}.
\]
This function satisfies the desired conditions. \end{proof}
\begin{cor}
\label{cor-caracteristic functions of constancy regions are p-fractal}Let
$\eta_{\bs[c]}$ be the characteristic function associated to the
constancy region of $\tau(\mf[a]^{\bs[c]})$, then $\eta_{\bs[c]}$
is a $p$- fractal. 
\end{cor}

\section{An Example}

In section 4 we showed that the characteristic functions of the constancy
regions are $p$-fractal functions, Corollary \ref{cor-caracteristic functions of constancy regions are p-fractal}.
We use this fact and Proposition \ref{thm: root ideals} to compute
an explicit example. Throughout this section we use a subscript $*_{p}$
to denote that the number is written in base $p$. One of the main
tools for computing examples is the following theorem:
\begin{thm}
\label{thm:(Lucas'-Theorem)}(Lucas' Theorem \cite{E}) Fix non-negative
integers $m\geq n\in\mathbb{N}$ and a prime number $p$. Write $m$
and $n$ in their base $p$ expansions: $m=\sum_{j=0}^{r}m_{j}p^{j}$
and $n=\sum_{j=0}^{r}n_{j}p^{j}.$Then modulo $p,$
\[
\binom{m}{n}=\binom{m_{0}}{n_{0}}\cdot\binom{m_{1}}{n_{1}}\cdots\binom{m_{r}}{n_{r}},
\]
where we interpret $\binom{a}{b}$ as zero if $a<b$. In particular,
$\binom{m}{n}$ is non-zero mod $p$ if and only if $m_{j}\geq n_{j}$
for all $j=1,\ldots r$. \end{thm}
\begin{rem}
\label{nonzero coefficients}In particular if $m=p^{k}-1$ all the
coefficients in the expansion of $(x+y)^{m}$ are nonzero. \end{rem}
\begin{example}
(The Devil's Staircase) Let $R=\mathbb{F}_{3}[x,y]$, $f_{1}=x+y$,
and $f_{2}=xy$. We want to describe the constancy regions for the
test ideals $\tau(f^{\boldsymbol{c}})$. 

We first show that there are five different test ideals in the region
$[0,1]\times[0,1]$. More precisely, we show that 
\[
\tau(f^{\boldsymbol{c}})=\begin{cases}
R\mbox{ or }(x,y), & \mbox{\ensuremath{\boldsymbol{c}\in[0,1)\times[0,1)}}\\
(x+y), & \boldsymbol{c}\in\{1\}\times[0,1)\\
(xy), & \boldsymbol{c}\in[0,1)\times\{1\}\\
(xy(x+y)), & \boldsymbol{c}=(1,1).
\end{cases}
\]

We want to compute the test ideal at $(\frac{1}{3},\frac{2}{3})$
. By Proposition \ref{Mixed test ideals =00003D normal test ideals}
\[
\tau(f^{(0.1_{3},0.2_{3})})=\tau((f_{1}\cdot f_{2}^{2})^{\frac{1}{3}}).
\]
 By Proposition \ref{Computing test ideals for principal ideals},
\[
\tau((f_{1}\cdot f_{2}^{2})^{\frac{1}{3}})=((x+y)(xy)^{2})^{[\frac{1}{3}]}=(x^{3}y^{2}+x^{2}y^{3})^{[\frac{1}{3}]}.
\]
Finally, Proposition \ref{thm: root ideals} gives 
\[
(x^{3}y^{2}+x^{2}y^{3})^{[\frac{1}{3}]}=(x,y),
\]
and therefore 
\[
\tau(f_{1}^{c_{1}}\cdot f_{2}^{c_{2}})\subseteq(x,y)\mbox{ if \ensuremath{c_{1}\geq1/3}and \ensuremath{c_{2}\geq2/3}.}
\]
In particular, the test ideal associated to the points $(1-\frac{1}{3^{k}},1-\frac{1}{3^{k}})$
is contained in $(x,y)$. Now 
\[
\tau(f^{(1-\frac{1}{3^{k}},1-\frac{1}{3^{k}})})=((x+y)^{3^{k}-1}(xy)^{3^{k}-1})^{[\frac{1}{3^{k}}]}
\]
\[
=((x^{2}y+xy^{2})^{3^{k}-1})^{[\frac{1}{3^{k}}]}.
\]

Since the terms $x^{2(3^{k}-1)}y^{3^{k}-1}$and $x^{3^{k}-1}y^{2(3^{k}-1)}$appear
in the expansion of $(x^{2}y+xy^{2})^{3^{k}-1}$ with nonzero coefficient,
Remark \ref{nonzero coefficients}. We conclude that $\tau(f^{(1-\frac{1}{3^{k}},1-\frac{1}{3^{k}})})\supseteq(x,y).$
Therefore 
\[
\tau(f^{(1-\frac{1}{3^{k}},1-\frac{1}{3^{k}})})=(x,y).
\]

Thus there are only two test ideals in the region $[0,1)\times[0,1)$,
these are $R$ and $(x,y)$.

Clearly $\tau(f^{(1,0)})=(x+y)$, and by Skoda's Theorem 
\[
\tau(f^{(1,1-\frac{1}{3^{k}})})=f_{1}\cdot\tau(f^{(0,1-\frac{1}{3^{k}})})
\]
\[
=(x+y)\cdot((xy)^{3^{k}-1})^{[\frac{1}{3^{k}}]}
\]
\[
=(x+y),
\]

hence the only test ideal in the region $[0,1)\times\{1\}$ is $(x+y)$.

In a similar way, $\tau(f^{(0,1)})=(xy)$ and 
\[
\tau(f^{(1-\frac{1}{3^{k}},1)})=f_{2}\cdot\tau(f^{(1-\frac{1}{3^{k}},0)})
\]
\[
=(xy)\cdot((x+y)^{3^{k}-1})^{[\frac{1}{3^{k}}]}
\]
\[
=(xy).
\]

Thus $(xy)$ is the only test ideal that appears in the region$\{1\}\times[0,1)$.

Lastly, note that the test ideal at $(1,1)$ is 
\[
\tau(f^{(1,1)})=((x+y)xy).
\]

We now show that $(\frac{1}{3},\frac{2}{3})$ is a point in the boundary
of $B^{(x,y)}(f)$ and then use the $p$-fractal structure to sketch
the constancy regions. 

For every $k$ 
\[
\tau(f^{(\frac{1}{3}-\frac{1}{3^{k}},\frac{2}{3}-\frac{1}{3^{k}})})
\]
\[
=((x+y)^{3^{k-1}-1}(xy)^{2\cdot3^{k-1}-1})^{[\frac{1}{3^{k}}]}.
\]
But in the expansion of $(x+y)^{3^{k-1}-1}$ every term appears with
nonzero coefficient, Remark \ref{nonzero coefficients}. In particular
the term $(xy)^{\frac{3^{k-1}-1}{2}}(xy)^{2\cdot3^{k-1}-1}$ appears
with non-zero coefficient when expanding the product $(x+y)^{3^{k-1}-1}(xy)^{2\cdot3^{k-1}-1}$.
Since the degrees in $x$ and $y$ of this monomial are smaller than
$3^{k}$, by Proposition \ref{thm: root ideals} we conclude that
$\tau(f^{(\frac{1}{3}-\frac{1}{3^{k}},\frac{2}{3}-\frac{1}{3^{k}})})=R$
. Thus 
\[
\chi_{f}^{(x,y)}(\frac{1}{3},\frac{2}{3})=0
\]

and 
\[
\chi_{f}^{(x,y)}([0,\frac{1}{3})\times[0,\frac{2}{3}))=1.
\]

The later shows that the point $(\frac{1}{3},\frac{2}{3})$ is in
the boundary of constancy regions for $R$ and $(x,y)$. We can use
the $p$-fractal structure to find more points in this boundary. The
idea is to break the region $[0,1]\times[0,1]$ into squares of length
$1/3$ and find which of these must contain a boundary point. Then
we apply the $p$-fractal structure to these squares to find the points. 

For the points $(0,\frac{2}{3})$, $(\frac{2}{3},\frac{1}{3})$, $(\frac{1}{3},1)$,
and $(1,\frac{2}{3})$ we have:
\[
\tau(f^{(0,\frac{2}{3})})=((xy)^{2})^{[\frac{1}{3}]}=R,
\]
\[
\tau(f^{(\frac{2}{3},\frac{1}{3})})=((x+y)^{2}xy)^{[\frac{1}{3}]}=(x^{3}y-x^{2}y^{2}+xy^{3})^{[\frac{1}{3}]}=R
\]

and 
\[
\tau(f^{(\frac{1}{3},1)})=((x+y)(xy)^{3})^{[\frac{1}{3}]}=(xy)\subset(x,y),
\]
\[
\tau(f^{(1,\frac{2}{3})})=((x+y)^{3}x^{2}y^{2})^{[\frac{1}{3}]}=(x+y)\subset(x,y).
\]

Therefore there should be boundary points in the squares $[0,1/3)\times[2/3,1)$
and $[2/3,1)\times[0,1/3)$. Is easy to check that there are not boundary
points in all the other squares. From this and lemma \ref{characterisct functions fractal behavior}
we know that $T_{3|(0,2)}\chi_{f}^{(x,y)}=T_{3|(2,1)}\chi_{f}^{(x,y)}=\chi_{f}^{(x,y)}$,
since $\chi_{f}^{(x,y)}$ is the only characteristic function that
is non constant in $[0,1)\times[0,1)$. Moreover,
\[
\chi_{f}^{(x,y)}(0.01_{3},0.22_{3})=\chi_{f}^{(x,y)}(0_{3}+0.01_{3},0.2_{3}+0.02_{3})
\]
\[
=\chi_{f}^{(x,y)}(\frac{0_{3}+0.1_{3}}{3},\frac{2_{3}+0.2_{3}}{3})=T_{3|(0,2)}\chi_{f}^{(x,y)}(0.1_{3},0.2_{3})
\]
\[
=\chi_{f}^{(x,y)}(0.1_{3},0.2_{3})=\chi_{f}^{(x,y)}(\frac{1}{3},\frac{2}{3})=0
\]
in a similar way 
\[
\chi_{f}^{(x,y)}(0.21_{3},0.12_{3})=0
\]
and 
\[
\chi_{f}^{(x,y)}([0,0.01_{3})\times[0,0.22_{3}))=\chi_{f}^{(x,y)}([0,0.21_{3})\times[0,0.12_{3}))=1.
\]
This is the points $(0.01_{3},0.22_{3})$ and $(0.21_{3},0.12_{3})$
are also in the boundary. We can repeat the proccess by subdividing
the squares $[0,1/3)\times[2/3,1)$ and $[2/3,1)\times[0,1/3)$ into
smaller squares of length $1/9$ and obtain more points of the boundary.
This process can be sumarized as follows. Let $A$ is the set of points
obtained from $(0.1_{3},0.2_{3})$ by successively applying the operations
\[
(0.a_{1}\ldots a_{n}1_{3},0.b_{1}\ldots b_{n}2_{3})\mapsto\begin{cases}
(0.a_{1}\ldots a_{n}01_{3},0.b_{1}\ldots b_{n}22_{3})\\
(0.a_{1}\ldots a_{n}21,0.b_{1}\ldots b_{n}12_{3})
\end{cases}
\]
then
\[
\chi_{f}^{(x,y)}(\boldsymbol{p})=0
\]
and 
\[
\chi_{f}^{(x,y)}([\boldsymbol{0},\boldsymbol{p}))=1
\]
for all $\boldsymbol{p}\in A$. This is, the points of $A$ are points
in the boundary. We can now sketch the regions of constancy in $[0,1]\times[0,1]$:
\end{example}
\begin{center}

\includegraphics[scale=0.28]{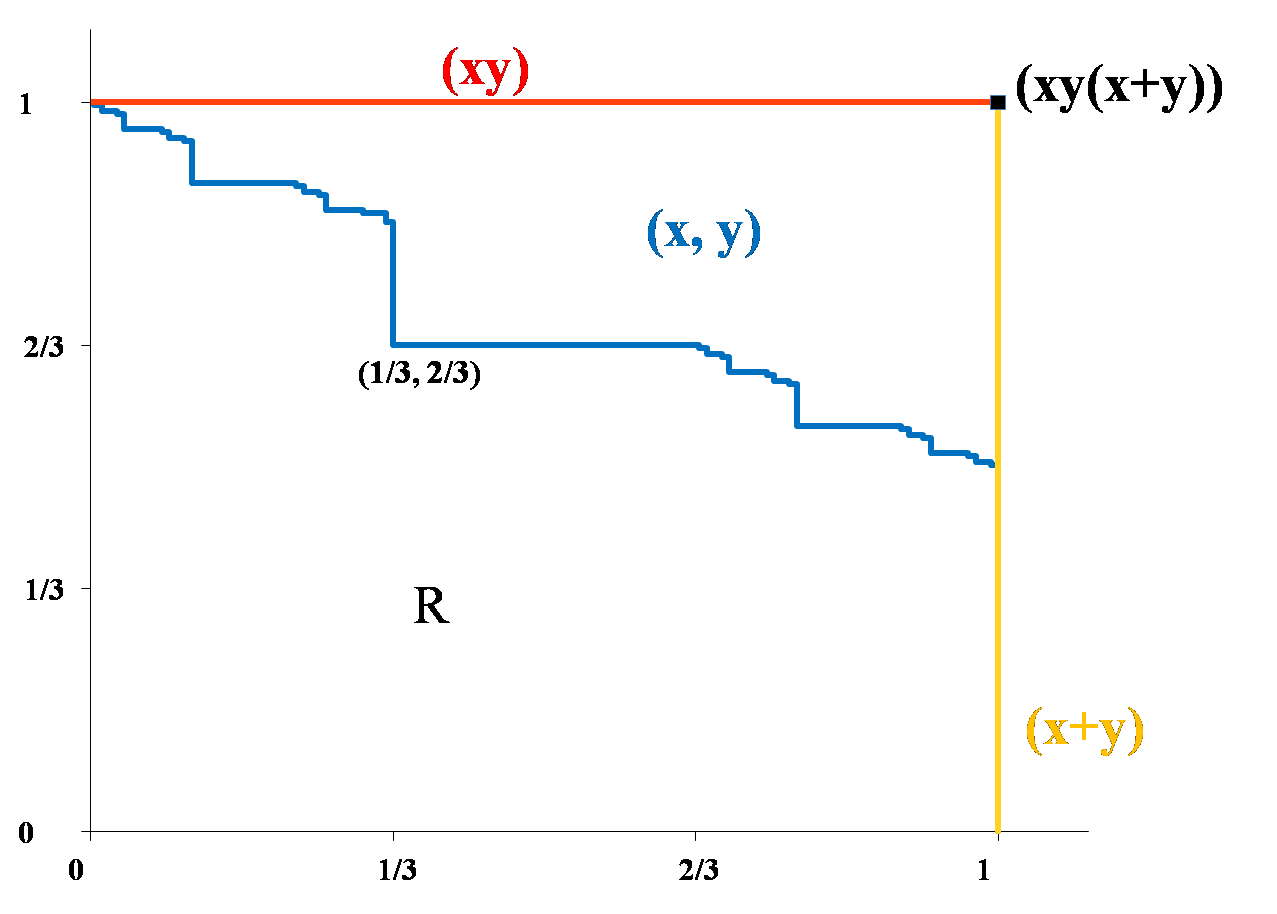}\end{center}

Using Skoda's theorem, we can describe the whole diagram of test ideals:

\begin{center}

\includegraphics[scale=0.28]{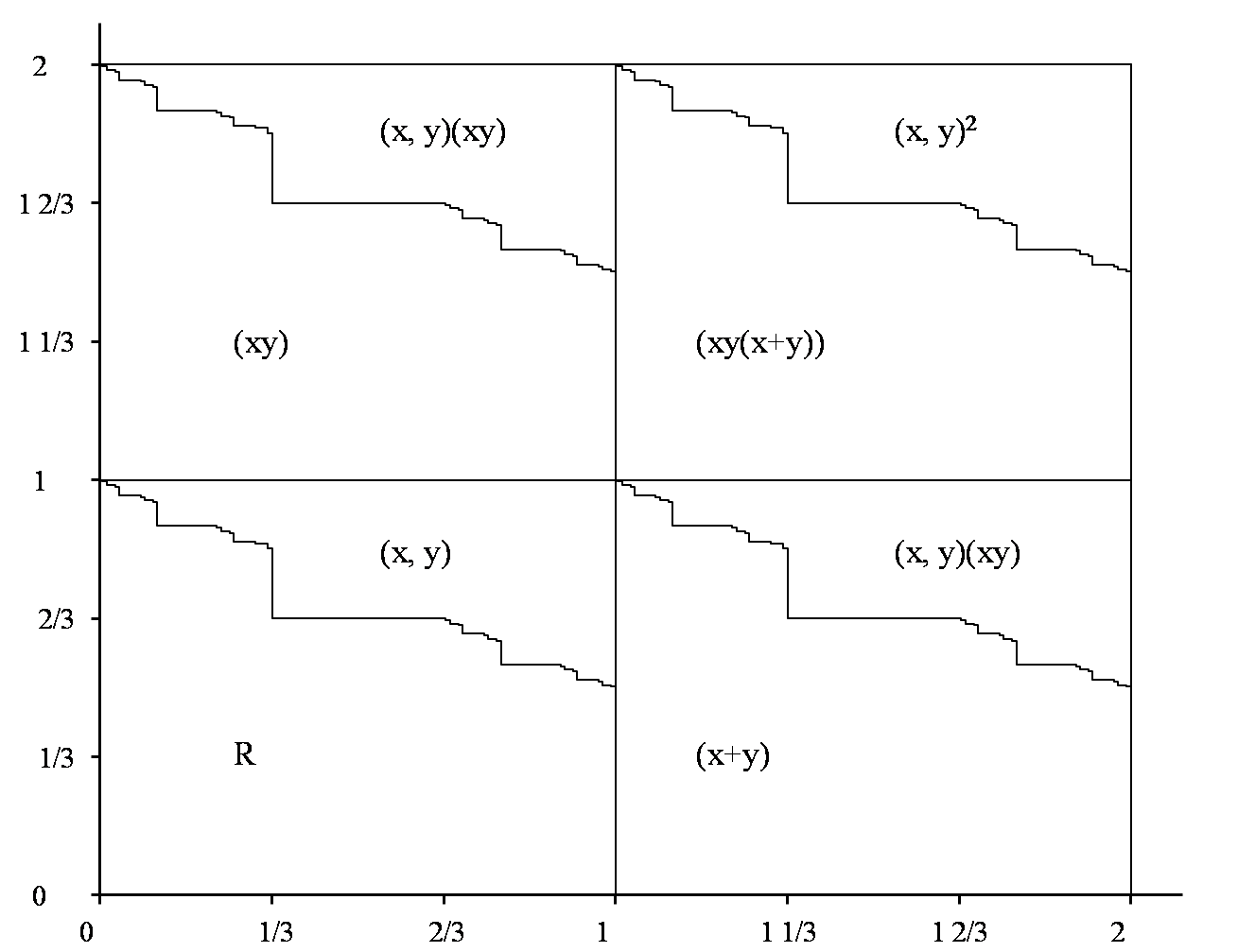}

\end{center}
\begin{rem}
We choose the name Devil's Staircase for this example, because the
resemblance to the Devil's Staircases or Cantor functions that appear
in the basic courses of analysis. \end{rem}
\begin{example}
In a similar way, it can be shown that for any characteristic $p$
the same polynomials give a staircase that has infinitely many steps.
Indeed, 
\[
\tau(f^{(\frac{1}{p^{k}},1-\frac{1}{p^{k}})})=((x+y)(xy)^{p^{k}-1})^{[\frac{1}{p^{k}}]}=(x,y)
\]

but 
\[
\tau(f^{(\frac{2}{p^{k}},1-\frac{2}{p^{k}})})=((x+y)^{2}(xy)^{p^{k}-2})^{[\frac{1}{p^{k}}]}=R
\]
and so we have many different points in the line $x+2y=2$ with test
ideal equal to $(x,y)$ and infinitely many with test ideal equal
to $R$. Therefore we can not expect that there are characteristics
for which the region given by the test ideals will be the same as
the one given by the multiplier ideals\end{example}

\address{Department of Mathematics, University of Michigan, 530 Church Street,
Ann Arbor, MI 48109, USA}

\medskip{}

\email{\textit{E-mail address: }juanfp@umich.edu}
\end{document}